\documentclass[12pt]{article}%
\usepackage{amsfonts}
\usepackage{amssymb}
\usepackage{hyperref}
\usepackage{url}
\usepackage{amsmath}%
\usepackage{graphicx}

\providecommand{\U}[1]{\protect\rule{.1in}{.1in}}

\newtheorem{theorem}{Theorem}

\newtheorem{lemma}[theorem]{Lemma}

\newenvironment{proof}[1][Proof]{\noindent\textbf{#1.} }{{\hfill $\Box$ \\}}
\begin{document}

\title{The groups and nilpotent Lie rings of order $p^{8}$ with maximal class}
\author{Seungjai Lee and Michael Vaughan-Lee}
\date{September 2021}
\maketitle

\begin{abstract}
We classify the nilpotent Lie rings of order $p^{8}$ with maximal class for
$p\geq5$. This also provides a classification of the groups of order $p^{8}$
with maximal class for $p\geq11$ via the Lazard correspondence.

\end{abstract}

\section{Introduction}

We give a classification of the nilpotent Lie rings of order $p^{8}$
$(p\geq5)$ which are of maximal class (i.e. nilpotent of class 7). The
classification gives us the following theorem.

\begin{theorem}
For $p\geq5$ the number of nilpotent Lie rings of order $p^{8}$ which have
maximal class is%
\begin{align*}
& 4p^{3}+7p^{2}+9p+6+(6p+11)\gcd(p-1,3)+4\gcd(p-1,5)\\
& +(p+2)\gcd(p-1,7)+(p+3)\gcd(p-1,8)+2\gcd(p-1,9)+\gcd(p-1,12).
\end{align*}

\end{theorem}

By the Lazard correspondence between $p$-groups and nilpotent Lie rings, for
$p\geq11$ this formula also gives us the number of groups of order $p^{8}$
with maximal class.

We have constructed a database of the nilpotent Lie rings of order $p^{8}$ with maximal
class which will be included in the next release of the \textsf{GAP} package 
LiePRing \cite{eickvl2015}. The LiePRing package can then be used to provide 
complete lists of the nilpotent Lie rings of order $p^{8}$ with maximal class 
for any given $p\geq5$, and can also be used to provide complete lists of the 
groups of order $p^{8}$ with maximal class for any given $p\geq11$. Both \textsf{GAP}
\cite{GAP4} and \textsc{Magma} \cite{boscan95} have databases of the groups of
order $2^{8}$ and $3^{8}$, and it is easy to extract the groups of maximal
class from these databases. For $p=5,7$ we can use the databases of groups of
order $p^{7}$ to obtain a list of the groups of maximal class, and then
use the Descendants function in \textsf{GAP} or \textsc{Magma} to compute the groups of order
$p^{8}$ with maximal class. (If $P$ has order $p^{8}$ and maximal class then
$P$ is a \textquotedblleft descendant\textquotedblright\ of the quotient $P/Z$
where $Z$ is the centre of $P$, and $P/Z$ is a group of order $p^{7}$ of
maximal class.)

\section{Preliminaries}

Let $L$ be a nilpotent Lie ring of order $p^{n}$ ($n\geq3$) and class $n-1$.
Here \textquotedblleft class\textquotedblright\ means nilpotency class, and
not $p$-class, which is the length of lower $p$-central series of
$L$. So $L$ is a nilpotent Lie ring of maximal class. Let the lower central
series of $L$ be%
\[
L>L^{2}>L^{3}>\ldots>L^{n-1}>L^{n}=\{0\},
\]
where $L^{2}=\langle ab\,|\,a,b\in L\rangle$, and where for $i>2$,
$L^{i}=\langle ab\,|\,a\in L^{i-1},\,b\in L\rangle$. (We denote the Lie
product of $a$ and $b$ by $ab$, rather than by $[a,b]$.) Then $L/L^{2}$ is
elementary abelian of order $p^{2}$, and for $2\leq i<n$ the quotient
$L^{i}/L^{i+1}$ has order $p$. Note that if $n>3$ then $L/L^{n-1}$ has maximal
class $n-2$.

\begin{lemma}
$pL\leq L^{n-1}$.
\end{lemma}

\begin{proof}
The proof is by induction on $n$. Note that there is nothing to prove if $n=3
$. So assume that $n>3$, and assume by induction that $pL\leq L^{n-2}$. Let
$L$ be generated by $a,b$, so that we can assume that $pa$, $pb\in L_{n-2}$.
If $pa\notin L^{n-1}$ then $a$ centralizes $L^{n-2}$. Similarly if $pb\notin
L^{n-1}$ then $b$ centralizes $L^{n-2}$, and if $pa+pb\notin L^{n-1}$ then
$a+b$ centralizes $L^{n-2}$. Now if one or the other or both of $pa,pb$ lie
outside $L^{n-1}$ then at least two of $pa$, $pb$, $pa+pb$ lie outside
$L^{n-1}$, which implies that $L$ is generated by elements which centralize
$L^{n-2}$. Clearly this is impossible, and so $pa,pb\in L^{n-1}$.
\end{proof}

Now let $L$ be a nilpotent Lie ring of order $p^{8}$ and maximal class 7. Then
$L/L^{7}$ is a nilpotent Lie ring of order $p^{7}$ with maximal class. By
Lemma 2, $L/L^{7}$ has characteristic $p$. The \textsf{GAP} package LiePRing
contains a database of the nilpotent Lie rings of order $p^{7}$ ($p\geq5$).
There are $p+8$ Lie rings of maximal class and characteristic $p$ in the
database. These are as follows. (Here $ba^3$ denotes $baaa$, $ba^4$ denotes
$baaaa$, and so on.)

\[
\langle a,b\,|\,bab,ba^3b,pa,pb,\,\text{class }6\rangle,\tag{7.623}%
\]
\[
\langle a,b\,|\,bab,ba^3b-ba^5,pa,pb,\,\text{class }6\rangle,\tag{7.627}%
\]
\[
\langle a,b\,|\,bab-ba^5,ba^3b,pa,pb,\,\text{class }6\rangle,\tag{7.633}%
\]
\[
\langle a,b\,|\,bab-ba^4,ba^3b,pa,pb,\,\text{class }6\rangle,\tag{7.641}%
\]
\[
\langle a,b\,|\,bab-ba^4,ba^3b-ba^5,pa,pb,\,\text{class }6\rangle ,\tag{7.646}%
\]
\[
\langle a,b\,|\,bab-ba^4-ba^5,ba^3b,pa,pb,\,\text{class }6\rangle ,\tag{7.648}%
\]
\[
\langle a,b\,|\,bab-ba^3,ba^3b-xba^5,pa,pb,\,\text{class }6\rangle\;(0\leq x<p),\tag{7.650}%
\]
\[
\langle a,b\,|\,bab-ba^3-ba^5,ba^3b-ba^5,pa,pb,\,\text{class }6\rangle,\tag{7.656}%
\]
\[
\langle a,b\,|\,bab-ba^3-wba^5,ba^3b-ba^5,pa,pb,\,\text{class }6\rangle.\tag{7.657}%
\]

The numbering 7.623, 7.627, \ldots\ gives the \textquotedblleft
LibraryName\textquotedblright\ of these Lie rings in the database. The
parameter $w$ in 7.657 is taken to be a (fixed) primitive element mod $p$.
Note that $p$ in these presentations can be replaced by 5 to give a complete
list of the nilpotent Lie rings of order $5^7$ of maximal class and characteristic 5, 
replaced by 7 to give a complete list of the nilpotent Lie rings of order $7^7$ of 
maximal class and characteristic 7, and so on.

\section{Computing descendants}

We use the \textit{Lie ring generation algorithm} as described in
\cite{newobvl} and \cite{obrienvl2} to compute the descendants of order
$p^{8}$ of the nilpotent Lie rings of order $p^{7}$ with maximal class. This
algorithm is an analogue of the $p$-group generation algorithm described in
\cite{OBrien90}. The Lie ring generation algorithm makes use of the lower
$p$-central series of a Lie ring $L$, which is defined in an an analogous way
to groups. We define the series
\[
L=L_{1}\geq L_{2}\geq L_{3}\geq\ldots\geq L_{c}\geq\ldots
\]
by setting $L_{1}=L$, $L_{2}=L^{2}+pL$, and for $c>1$ we set $L_{c+1}%
=L_{c}L+pL_{c}$. (Here $L_{c}L$ is $\langle ab\,|\,a\in L_{c},\,b\in L\rangle
$.) Note that we use superscripts to denote terms of the lower central series,
and subscripts to denote terms of the lower $p$-central series. In the case of
a nilpotent Lie ring of order $p^{n}$ with maximal class the two series are
identical. The ideal $L_{c}$ consists of all linear combinations of terms of
the form
\[
a_{1}a_{2}\ldots a_{c},\,pa_{1}a_{2}\ldots a_{c-1},\,p^{2}a_{1}a_{2}\ldots
a_{c-2},\ldots,\,p^{c-1}a_{1}.
\]
We say that $L$ has $p$-class $c$ if $L_{c+1}=\{0\}$, $L_{c}\neq\{0\}$.

If $L$ is a nilpotent Lie ring with finite order $p^{n}$ for some prime $p$,
then $L_{c+1}$ will equal $\{0\}$ for some $c$. In fact if $L $ is nilpotent
of class $k$, and if the exponent of $L$ as a finite abelian group is $p^{m}$
then $L$ has $p$-class $c$ for some $c$ with $k\leq c<k+m$.

If $L$ and $M$ are two finite nilpotent Lie rings with prime-power order, then
$L$ is a \textit{descendant} of $M$ if $L/L_{c}\cong M$ for some $c\geq2$. If
$L/L_{c}\cong M$ and $L$ has $p$-class $c$ (so that $L_{c}\neq\{0\}$,
$L_{c+1}=\{0\}$) then $L$ is an \textit{immediate descendant} of $M$. Note
that if $L$ is a descendant of $M$ then $L/L_{2}\cong M/M_{2}$, so that $L$
and $M$ have the same generator number.

If $M$ is a nilpotent $d$-generator Lie ring of order $p^{n}$, then we
construct its $p$-covering ring $\widehat{M}$. This is the largest $d$-generator Lie
ring $\widehat{M}$ having a central elementary abelian ideal $Z$ such that
$\widehat{M}/Z\cong M$ and every immediate descendant of $M$ is isomorphic to
$\widehat{M}/T$ for some $T\leq Z$. However $\widehat{M}/T$ is not an
immediate descendant of $M$ for every subring $T\leq Z$. If $M$ has $p$-class
$c$ (so that $M_{c+1}=\{0\}$) then we define the \emph{nucleus} of $M$ to be
$\widehat{M}_{c+1}$. Then $\widehat{M}/T$ is an immediate descendant of $M$ if
and only if $T$ is a proper subring of $Z$ such that $T$ supplements the
nucleus $\widehat{M}_{c+1}$. It can happen that $\widehat{M}_{c+1}=\{0\}$, in
which case $M$ has no immediate descendants and is \textit{terminal}.

Hence we obtain a complete list of the immediate descendants of $M$ by
calculating its $p$-covering ring $\widehat{M}$, and listing the proper
subrings $T<Z$ such that $T+\widehat{M}_{c+1}=Z$. (These are the
\textit{allowable subrings} of $Z$.)

We now have a list of the immediate descendants of $M$, and we can easily
restrict to those with a specified order. This list will usually contain
redundancies, and we need to solve the isomorphism problem. This is done as
follows. We compute the automorphism group of $M$ and we extend each
automorphism $\alpha$ of $M$ to an automorphism $\alpha^{\ast}$ of
$\widehat{M}$. (If $M$ is generated by $a_{1},a_{2},\ldots,a_{d}$ then we
choose preimages $x_{1},x_{2},\ldots,x_{d}$ in $\widehat{M}$ for $a_{1}%
,a_{2},\ldots,a_{d}$, and preimages $y_{1},y_{2},\ldots,y_{d}$ in $\widehat
{M}$ for $a_{1}\alpha,a_{2}\alpha,\ldots,a_{d}\alpha$. Then $x_{1}%
,x_{2},\ldots,x_{d}$ generate $\widehat{M}$, and we define $\alpha^{\ast}$ by
setting $x_{i}\alpha^{\ast}=y_{i}$ for $i=1,2,\ldots,d$.) Then $Z\alpha^{\ast
}=Z$, and the action of $\alpha^{\ast}$ on $Z$ is uniquely determined by
$\alpha$. Two allowable subrings $T_{1},T_{2}$ define isomorphic descendants
$\widehat{M}/T_{1},\widehat{M}/T_{2}$ if and only if $T_{2}\alpha^{\ast}%
=T_{1}$ for some automorphism $\alpha$ of $M$. We obtain a complete
irredundant set of immediate descendants of $M$ by choosing a set of
representatives for the orbits of the allowable subrings of $Z$ under this
action of the automorphism group of $M$.

\section{The nilpotent Lie rings of order $p^{8}$ with maximal class}

In this section we give a complete list of presentations for the nilpotent Lie
rings of order $p^{8}$ with maximal class ($p\geq5$). Many of the
presentations involve parameters $w,x,y,z$. These parameters take
integer values in the range $0,1,\ldots,p-1$. The parameter $w$ is always
assumed to be a (fixed) primitive element mod $p$. Associated with the
presentations are somewhat cryptic comments intended to describe when two sets
of parameters give isomorphic Lie rings. For example four of the descendants
of 7.623 have a single parameter $x$, together with the
comment \textquotedblleft$x \neq 0$, $x\sim xa^{6}$\textquotedblright. Here
(and for all these comments) $a$ is assumed to range over all integers which
are non-zero modulo $p$. So this comment is intended to mean that $x$ can take
any value in the range $1,2,\ldots,p-1$ and that if $0<x,y<p$ then $x$ and $y$
give isomorphic Lie rings if and only $x=ya^{6}\operatorname{mod}p$ for some
integer $a$ which is not divisible by $p$. Actually there is no reason to
restrict $x$ to the range $1,2,\ldots,p-1$ since $x$ is the coefficient of
an element of order $p$ in the presentations of these Lie rings. So if 
$x=x'\operatorname{mod}p$ then $x$ and $x'$ give identical Lie rings.

The simplest way to
\textquotedblleft solve\textquotedblright\ these conditions is to treat them
as defining equivalence relations over GF$(p)$. Two non-zero elements in
GF($p$) give isomorphic algebras if and only if they lie in the same coset of
the subgroup $\{a^{6}\,|\,a\in$\thinspace GF$(p)^{\ast}\}$ of the
multiplicative group GF$(p)^{\ast}$ of non-zero elements in GF$(p)$. This
subgroup has order $\frac{p-1}{6}$ if $\gcd(p-1,3)=3$ and order $\frac{p-1}%
{2}$ if $\gcd(p-1,3)=1$. If we let $u$ be a primitive element in GF($p$) then
$1,u,u^{2},u^{3},u^{4},u^{5}$ is a transversal for this subgroup when
$\gcd(p-1,3)=3$, and $1,u$ is a transversal for the subgroup if $\gcd
(p-1,3)=1$. So we obtain a complete and irredundant set of representatives for
the isomorphism classes of these Lie rings as $x$ ranges over $1,2,\ldots,p-1$
by taking $x=1,w,w^{2}\operatorname{mod}p,w^{3}\operatorname{mod}%
p,w^{4}\operatorname{mod}p,w^{5}\operatorname{mod}p$ when $\gcd(p-1,3)=3$ and
taking $x=1,w$ when $\gcd(p-1,3)=1$.

The comments associated with the other Lie rings with a single parameter $x$
are similar. One of the descendants of 7.627 has two parameters $x,y$ with the 
comment \textquotedblleft$x,y \neq 0$, $[x,y]\sim\lbrack xa^{7},ya^{6}]$%
\textquotedblright. We can solve this over GF$(p)$ by letting $y$ range over a
transversal for the subgroup $\{a^{6}\,|\,a\in$\thinspace GF$(p)^{\ast}\}$ of
the group GF$(p)^{\ast}$. For any given $y$ in this transversal the pair
$[x,y]$ gives an isomorphic Lie ring to the pair $[x^{\prime},y]$ if and only
if $x^{\prime}=xa$ for some $a\in\,$GF$(p)^{\ast}$ satisfying $a^{6}=1$. We
can solve this equivalence relation on the values for $x$ over GF$(p)$, and
lift the a set of representatives for the equivalence classes to integers in
the range $1,2,\ldots,p-1$.

All the Lie rings in the list below have nilpotency class 7, but we leave the 
class unspecified, to save space.

\subsection{The descendants of 7.623}

7.623 has $8+8\gcd(p-1,3)+2\gcd(p-1,5)+\gcd(p-1,8)$ descendants of order $p^8$.
\[
\langle a,b\,|\,bab,ba^3b,pa,pb,ba^5b\rangle
\]
\[
\langle a,b\,|\,bab,ba^3b,pa-ba^6,pb,ba^5b\rangle
\]
\[
\langle a,b\,|\,bab,ba^3b,pa,pb-xba^6,ba^5b\rangle \;(x \neq 0, x \sim xa^6)
\]
\[
\langle a,b\,|\,bab,ba^3b-ba^6,pa,pb,ba^5b\rangle
\]
\[
\langle a,b\,|\,bab,ba^3b-ba^6,pa-xba^6,pb,ba^5b\rangle \;(x \neq 0, x \sim xa^8)
\]
\[
\langle a,b\,|\,bab,ba^3b-ba^6,pa,pb-xba^6,ba^5b\rangle \;(x \neq 0, x \sim xa^6)
\]
\[
\langle a,b\,|\,bab-ba^6,ba^3b,pa,pb,ba^5b\rangle
\]
\[
\langle a,b\,|\,bab-ba^6,ba^3b,pa-xba^6,pb,ba^5b\rangle \;(x \neq 0, x \sim xa^{10})
\]
\[
\langle a,b\,|\,bab-ba^6,ba^3b,pa,pb-xba^6,ba^5b\rangle \;(x \neq 0, x \sim xa^6)
\]
\[
\langle a,b\,|\,bab,ba^3b,pa,pb,ba^6\rangle
\]
\[
\langle a,b\,|\,bab,ba^3b,pa-ba^5b,pb,ba^6\rangle
\]
\[
\langle a,b\,|\,bab,ba^3b,pa-wba^5b,pb,ba^6\rangle
\]
\[
\langle a,b\,|\,bab,ba^3b,pa,pb-ba^5b,ba^6\rangle
\]
\[
\langle a,b\,|\,bab,ba^3b,pa-xba^5b,pb-ba^5b,ba^6\rangle \;(x \neq 0, x \sim xa^6)
\]

\subsection{Descendants of 7.627}

7.627 has $p+2\gcd(p-1,3)+\gcd(p-1,7)$ descendants of order $p^8$.
\[
\langle a,b\,|\,bab,ba^3b-ba^5,pa,pb\rangle 
\]
\[
\langle a,b\,|\,bab,ba^3b-ba^5,pa-xba^6,pb\rangle \;(x \neq 0, x \sim xa^7)
\]
\[
\langle a,b\,|\,bab,ba^3b-ba^5,pa,pb-xba^6\rangle \;(x \neq 0, x \sim xa^6)
\]
\[
\langle a,b\,|\,bab,ba^3b-ba^5,pa-xba^6,pb-yba^6\rangle \;(x,y \neq 0, [x,y] \sim [xa^7,ya^6])
\]

\subsection{Descendants of 7.633}

7.633 has $4p+(p+1)\gcd(p-1,3)+2\gcd(p-1,9)+\gcd(p-1,12)$ descendants of order $p^8$.
\[
\langle a,b\,|\,bab-ba^5,ba^3b,pa,pb,ba^5b\rangle 
\]
\[
\langle a,b\,|\,bab-ba^5,ba^3b,pa,pb-xba^6,ba^5b\rangle \;(x \neq 0, x \sim xa^6)
\]
\[
\langle a,b\,|\,bab-ba^5,ba^3b,pa-xba^6,pb,ba^5b\rangle \;(x \neq 0, x \sim xa^9)
\]
\[
\langle a,b\,|\,bab-ba^5-ba^6,ba^3b,pa,pb,ba^5b\rangle 
\]
\[
\langle a,b\,|\,bab-ba^5-ba^6,ba^3b,pa-xba^6,pb,ba^5b\rangle \;(x \neq 0)
\]
\[
\langle a,b\,|\,bab-ba^5-ba^6,ba^3b,pa,pb-xba^6,ba^5b\rangle \;(x \neq 0)
\]
\[
\langle a,b\,|\,bab-ba^5,ba^3b-ba^6,pa,pb,ba^5b\rangle 
\]
\[
\langle a,b\,|\,bab-ba^5,ba^3b-ba^6,pa-xba^6,pb,ba^5b\rangle \;(x \neq 0)
\]
\[
\langle a,b\,|\,bab-ba^5,ba^3b-ba^6,pa,pb-xba^6,ba^5b\rangle \;(x \neq 0)
\]
\[
\langle a,b\,|\,bab-ba^5,ba^3b,pa,pb,ba^6\rangle 
\]
\[
\langle a,b\,|\,bab-ba^5,ba^3b,pa-xba^5b,pb,ba^6\rangle \;(x \neq 0, x \sim xa^{12})
\]
\[
\langle a,b\,|\,bab-ba^5,ba^3b,pa,pb-xba^5b,ba^6\rangle \;(x \neq 0, x \sim xa^9)
\]
\[
\langle a,b\,|\,bab-ba^5,ba^3b,pa-xba^5b,pb-yba^5b,ba^6\rangle \;(x,y \neq 0, [x,y] \sim [xa^{12},ya^9])
\]

\subsection{Descendants of 7.641}

7.641 has $5p-3+2p\gcd(p-1,3)+2\gcd(p-1,5)+(p+1)\gcd(p-1,8)$ descendants of order $p^8$.
\[
\langle a,b\,|\,bab-ba^4,ba^3b-xba^6,pa,pb,ba^5b\rangle \;(\text{all }x)
\]
\[
\langle a,b\,|\,bab-ba^4,ba^3b-xba^6,pa-yba^6,pb,ba^5b\rangle \;(y \neq 0, [x,y] \sim [x,ya^8])
\]
\[
\langle a,b\,|\,bab-ba^4,ba^3b-xba^6,pa,pb-yba^6,ba^5b\rangle \;(y \neq 0, [x,y] \sim [x,ya^6])
\]
\[
\langle a,b\,|\,bab-ba^4,ba^3b-ba^6,pa-xba^6,pb-yba^6,ba^5b\rangle \;(x,y \neq 0, [x,y] \sim [xa^8,ya^6])
\]
\[
\langle a,b\,|\,bab-ba^4,ba^3b,pa,pb,ba^6\rangle 
\]
\[
\langle a,b\,|\,bab-ba^4,ba^3b,pa-xba^5b,pb,ba^6\rangle \;(x \neq 0, x \sim xa^{10})
\]
\[
\langle a,b\,|\,bab-ba^4,ba^3b,pa,pb-xba^5b,ba^6\rangle \;(x \neq 0, x \sim xa^8)
\]
\[
\langle a,b\,|\,bab-ba^4,ba^3b,pa-xba^5b,pb-yba^5b,ba^6\rangle \;(x,y \neq 0, [x,y] \sim [xa^{10},ya^8])
\]

\subsection{Descendants of 7.646}

7.646 has $p^2$ descendants of order $p^8$.
\[
\langle a,b\,|\,bab-ba^4,ba^3b-ba^5,pa-xba^6,pb-yba^6\rangle \;(\text{all }x,y)
\]

\subsection{Descendants of 7.648}

7.648 has $4p^2-3p+1$ descendants of order $p^8$.
\[
\langle a,b\,|\,bab-ba^4-ba^5,ba^3b-xba^6,pa,pb,ba^5b\rangle \;(x \neq 1)
\]
\[
\langle a,b\,|\,bab-ba^4-ba^5,ba^3b-xba^6,pa-yba^6,pb,ba^5b\rangle \;(x \neq 1, y \neq 0)
\]
\[
\langle a,b\,|\,bab-ba^4-ba^5,ba^3b-xba^6,pa,pb-yba^6,ba^5b\rangle \;(x \neq 1, y \neq 0)
\]
\[
\langle a,b\,|\,bab-ba^4-ba^5,ba^3b-ba^6,pa-xba^6,pb-yba^6,ba^5b\rangle \;(\text{all }x,y)
\]
\[
\langle a,b\,|\,bab-ba^4-ba^5,ba^3b,pa-xba^5b,pb-yba^5b,ba^6\rangle \;(\text{all }x,y)
\]

\subsection{Descendants of 7.650}

7.650 is a family of $p$ Lie rings, and between them they have 
\[
2p^3+3p^2+p+3p\gcd(p-1,3)+(p+1)\gcd(p-1,7)+\gcd(p-1,8)
\]
descendants of order $p^8$.
\[
\langle a,b\,|\,bab-ba^3,ba^3b-ba^5,ba^5b,pa,pb\rangle 
\]
\[
\langle a,b\,|\,bab-ba^3,ba^3b-ba^5,ba^5b,pa-xba^6,pb\rangle \;(x \neq 0, x \sim xa^7)
\]
\[
\langle a,b\,|\,bab-ba^3,ba^3b-ba^5,ba^5b,pa,pb-xba^6\rangle \;(x \neq 0, x \sim xa^6)
\]
\[
\langle a,b\,|\,bab-ba^3,ba^3b-ba^5,ba^5b,pa-xba^6,pb-yba^6\rangle \;(x,y \neq 0, [x,y] \sim [xa^7,ya^6])
\]
\[
\langle a,b\,|\,bab-ba^3,ba^3b-ba^5-ba^6,ba^5b,pa-xba^6,pb-yba^6\rangle \;(\text{all }x,y)
\]
\[
\langle a,b\,|\,bab-ba^3-xba^6,ba^3b-ba^5,ba^5b,pa-yba^6,pb-zba^6\rangle \;(x \neq 0, [x,y,z] \sim [xa^3,ya^7,za^6])
\]
\[
\langle a,b\,|\,bab-ba^3,ba^3b-xba^5,pa,pb,ba^5b\rangle \;(x \neq 1)
\]
\[
\langle a,b\,|\,bab-ba^3,ba^3b-xba^5,pa-yba^6,pb,ba^5b\rangle \;(x \neq 1, y \neq 0, [x,y] \sim [x,ya^7])
\]
\[
\langle a,b\,|\,bab-ba^3,ba^3b-xba^5,pa,pb-yba^6,ba^5b\rangle \;(x \neq 1, y \neq 0, [x,y] \sim [x,ya^6])
\]
\[
\langle a,b\,|\,bab-ba^3,ba^3b-xba^5,pa-yba^6,pb-zba^6,ba^5b\rangle \;(x \neq 1, y,z \neq 0, [x,y,z] \sim [x,ya^7,za^6])
\]
\[
\langle a,b\,|\,bab-ba^3,ba^3b-xba^5-ba^6,pa-yba^6,pb-zba^6,ba^5b\rangle \;(x \neq 1,\text{all }y,z)
\]
\[
\langle a,b\,|\,bab-ba^3,ba^3b-3ba^5,pa,pb,ba^6\rangle 
\]
\[
\langle a,b\,|\,bab-ba^3,ba^3b-3ba^5,pa-xba^5b,pb,ba^6\rangle \;(x \neq 0, x \sim xa^8)
\]
\[
\langle a,b\,|\,bab-ba^3,ba^3b-3ba^5,pa,pb-xba^5b,ba^6\rangle \;(x \neq 0, x \sim xa^7)
\]
\[
\langle a,b\,|\,bab-ba^3,ba^3b-3ba^5,pa-xba^5b,pb-yba^5b,ba^6\rangle \;(x,y \neq 0, [x,y] \sim [xa^8,ya^7])
\]
\[
\langle a,b\,|\,bab-ba^3,ba^3b-3ba^5-ba^5b,pa-xba^5b,pb-yba^5b,ba^6\rangle \;([x,y] \sim [x,-y])
\]
\[
\langle a,b\,|\,bab-ba^3,ba^3b-3ba^5-wba^5b,pa-xba^5b,pb-yba^5b,ba^6\rangle \;([x,y] \sim [x,-y])
\]
\[
\langle a,b\,|\,bab-ba^3,ba^3b-3ba^5-xba^6,pa-yba^6,pb-zba^6,ba^5b-ba^6\rangle \;(\text{all }x,y,z)
\]

\subsection{Descendants of 7.656}

7.656 has $p^3 - \frac{p^2-p}{2}$ descendants of order $p^8$.
\[
\langle a,b\,|\,bab-ba^3-ba^5,ba^3b-ba^5,ba^5b,pa-xba^6,pb-yba^6\rangle \;([x,y] \sim [-x,y])
\]
\[
\langle a,b\,|\,bab-ba^3-ba^5-xba^6,ba^3b-ba^5,ba^5b,pa-yba^6,pb-zba^6\rangle \;(x \neq 0, [x,y,z] \sim [-x,y,z])
\]
\[
\langle a,b\,|\,bab-ba^3-ba^5,ba^3b-ba^5-xba^6,ba^5b,pa-yba^6,pb-zba^6\rangle \;(x \neq 0, [x,y,z] \sim [-x,y,z])
\]

\subsection{Descendants of 7.657}

7.657 has $p^3 - \frac{p^2-p}{2}$ descendants of order $p^8$.
\[
\langle a,b\,|\,bab-ba^3-wba^5,ba^3b-ba^5,ba^5b,pa-xba^6,pb-yba^6\rangle \;([x,y] \sim [-x,y])
\]
\[
\langle a,b\,|\,bab-ba^3-wba^5-xba^6,ba^3b-ba^5,ba^5b,pa-yba^6,pb-zba^6\rangle \;(x \neq 0, [x,y,z] \sim [-x,y,z])
\]
\[
\langle a,b\,|\,bab-ba^3-wba^5,ba^3b-ba^5-xba^6,ba^5b,pa-yba^6,pb-zba^6\rangle \;(x \neq 0, [x,y,z] \sim [-x,y,z])
\]

\section{Acknowledgements}

Our classification of the nilpotent Lie rings of order $p^8$ with maximal class ($p \ge 5$) is
essentially a hand calculation, with some computer assistance with \textsc{Magma} \cite{boscan95}.
We used Eamonn O'Brien's $p$-group generation algorithm \cite{OBrien90} in \textsc{Magma} to compute
the groups of order $p^8$ with maximal class for $11 \le p \le 43$, and confirmed that the number
of groups in these cases agreed with the PORC formula given in Theorem 1. We also used Serena Cical\`{o} 
and Willem de Graaf's implementation of the Lazard correspondence in their \textsf{GAP} package LieRing
\cite{SCWdG} to obtain the groups corresponding to the nilpotent Lie rings in our database, so
that we could compare them with the groups provided by the $p$-group generation algorithm. We used
Eamonn O'Brien's StandardPresentation function in \textsc{Magma} to prove that the two sets of groups 
are identical (up to isomorphism) for $11 \le p \le 43$.

The first author was supported by the National Research Foundation of Korea (NRF) grant funded by the Korean government (MEST), No. 2019R1A6A1A10073437.


\begin{thebibliography}{1}

\bibitem{boscan95}
W.~Bosma, J.~Cannon, and C.~Playoust, {\em The {M}agma algebra system {I}:
  {T}he user language}, J. Symbolic Comput. {\bf 24} (1997), 235--265.

\bibitem{SCWdG}
Serena Cical\`{o} and Willem de Graaf, {\em Lie{R}ing -- {Computing with finitely presented 
Lie rings}}, (2019), a \textsf{GAP} 4 package.

\bibitem{eickvl2015}
B.~Eick and M.~Vaughan-Lee, {\em Lie{PR}ing -- {Database and algorithms for
  {L}ie $p$-rings}}, (2015), a \textsf{GAP} 4 package.

\bibitem{GAP4}
The GAP~Group, {\em {GAP -- Groups, Algorithms, and Programming, Version
  4.11}}, (2020). Available from \url{http://www.gap-system.org}.

\bibitem{newobvl}
M.F. Newman, E.A. O'Brien, and M.R. Vaughan-Lee, {\em Groups and nilpotent
  {Lie} rings whose order is the sixth power of a prime}, J. Algebra {\bf 278}
  (2004), 383--401.

\bibitem{OBrien90}
E.A. O'{B}rien, {\em The $p$-group generation algorithm}, J. Symbolic
  Comput. {\bf 9} (1990), 677--698.

\bibitem{obrienvl2}
E.A. O'Brien and M.R. Vaughan-Lee, {\em The groups with order $p^7$ for odd
  prime $p$}, J. Algebra {\bf 292} (2005), 243--358.

\end{thebibliography}
\end{document}